\documentclass[letterpaper, 10 pt,conference]{ieeeconf}  % Comment this line out if you need a4paper

\IEEEoverridecommandlockouts                              % This command is only needed if 
\usepackage{graphicx}      % include this line if your document contains figures
\usepackage{graphics} % for pdf, bitmapped graphics files
\usepackage{epsfig} % for postscript graphics files
\usepackage{times} % assumes new font selection scheme installed
\usepackage{amsmath} % assumes amsmath package installed
\usepackage{amssymb}  % assumes amsmath package installed
\usepackage{caption}
\usepackage{subcaption}

\usepackage{algorithm}
\usepackage{algorithmicx}
\usepackage{algpseudocode}

\usepackage{tikz}
\usepackage{tkz-graph}
\usepackage{graphicx}

\renewcommand{\baselinestretch}{0.915}

\usepackage{cite}

\usepackage{balance}

\newtheorem{thm}{Theorem}
\newtheorem{lem}[thm]{Lemma}
\newtheorem{prop}[thm]{Proposition}

\newtheorem{problem}{Problem}

\newtheorem{remark}[thm]{Remark}

\newcommand{\N}{\mathcal{N}}

\newcommand{\R}{\mathbb{R}}

\newcommand{\T}{^{\top}}
\newcommand{\tr}{\mathrm{tr}}

\newcommand{\rev}[1]{ #1 }

\usepackage{hyperref}
\hypersetup{
	linkcolor=blue,
	filecolor=magenta,      
	urlcolor=red,
	pdftitle={Overleaf Example},
	pdfpagemode=FullScreen,
}

\newcommand{\blue}{\color{blue}}

\newif\iflongversion
\longversiontrue %÷ switch to true for long/full version, with proof
\title{\LARGE \bf 
Actuator Scheduling for Linear Systems: A Convex Relaxation  Approach 

}

\author{Junjie Jiao, Dipankar Maity, John S. Baras, and Sandra Hirche% <-this % stops a space
\thanks{This work was  supported in part by the European Union’s Horizon 2020 research and innovation programme under the Marie Sk\l{}odowska-Curie grant agreement no. 899987, in part by  the German
Research Foundation within the Priority Program SPP 1914 - Cyber-Physical Networking, and in part by a   DFG Mercator Fellowship.
}
\thanks{J. Jiao and S. Hirche are  with the Chair of Information-oriented Control,  Department of 	Electrical and Computer Engineering,  Technical University of Munich, Munich, Germany.
Email: 
{\tt junjie.jiao@tum.de; 
	hirche@tum.de}
	}%
\thanks{D. Maity is with the Department of Electrical and Computer Engineering, University of North Carolina at Charlotte, NC, 28223, USA. 
Email: 
{\tt dmaity@uncc.edu}}%
\thanks{J. S. Baras is with the Department of Electrical and Computer Engineering, University of Maryland at College Park, MD, 20742, USA. 
Email: 
{\tt baras@umd.edu}}%
}

\begin{document}

% \color{black}

\maketitle
\thispagestyle{empty}
\pagestyle{empty}

%%%%%%%%%%%%%%%%%%%%%%%%%%%%%%%%%%%%%%%%%%%%%%%%%%%%%%%%%%%%%%%%%%%%%%%%%%%%%%%%

% {\red As the problem has been modified, the abstract and conclusion have to be changed accordingly. Mentioning about actuator cost and multi-actuator case.

% Maybe stress that our results outperform the greedy algorithm.
% }
% \ifnewversion
\begin{abstract}

In this letter,	we investigate the problem of actuator scheduling for networked control systems. 
Given a stochastic linear system with a number of  actuators, we consider the case that    one actuator is activated at each time.
This problem is combinatorial in nature and NP hard to solve.
We propose a convex relaxation to the actuator scheduling problem, and use its solution as a {\em{reference}} to design an algorithm for solving the original scheduling problem. 
Using dynamic programming arguments, we provide a suboptimality bound of our proposed  algorithm.
%
% {\blue  
Furthermore, we show that our framework can be extended to incorporate multiple actuator scheduling at each time and  actuation costs.
% }
%
A simulation example is provided, which shows that our proposed method outperforms a random selection approach and a greedy selection approach.
%
%
% This is a dual problem to the sensor scheduling problem.
%
% We show that the actuator scheduling problem has a close relation to the linear quadratic Gaussian control problem, and that the solution of an actuator schedule problem can be extracted from an equivalent linear quadratic Gaussian control problem. 
% %
% We propose a convex relaxation to the linear quadratic Gaussian control problem and a reference covariance trajectory is obtained from solving the relaxed control problem.
% %
% Afterwards, a  \textit{covariance tracking} algorithm is designed to obtain an approximate solution to the actuator scheduling problem using the reference covariance trajectory obtained from the associated control problem.
% %
% While the actuator scheduling problem is NP-hard, the proposed framework circumvents this computational complexity by decomposing this problem into a convex linear quadratic Gaussian control problem and a covariance tracking problem.
% %
% We provide theoretical justification and a suboptimality bound for the proposed method using dynamic programming.
\end{abstract}

% \else
% \begin{abstract}
% \red
% In this paper,		we investigate the problem of actuator scheduling for {\blue stochastic} linear {\blue networked control} systems. 
% %
% Given a linear system with a number of  actuators, we consider the case that    one actuator is activated at each time.
% % 
% This problem is combinatorial in nature and NP hard to solve.
% %
% Theorefore, instead of finding a schedule directly, we solve this problem by investigating its connection to a linear quadratic Gaussian (LQG)  control problem with appropriate constraints. 
% %
% In particular, we  use the solution of this LQG control problem as a `reference',  and obtain an algorithm for solving the actuator scheduling problem. 
% %
% We show that our proposed  algorithm provides  suboptimality guarantees.
% %
% % {\blue  
% Furthermore, we also show that our results can   be extended to the cases with multiple actuator scheduling at each time and   actuation cost.
% % }
% %
%  The performance of the proposed algorithm is illustrated by a simulation example.
% \end{abstract}

% \fi


	\section{Introduction}\label{sec_intro}
% {\red 
%     The introduction goes here.
% {\red 
%      The introduction is to be modified/improved according to Sandra's comments.
% }

In recent years, networked control systems (NCSs) have gained much interest in the controls community due to the advancements in communication architecture, computer technology, and network infrastructure  that enable efficient distributed sensing, estimation, and control \cite{survey2006proc_NCS,walsh2001scheduling,baras1988sensor}.
%
% 			Potential applications  include smart buildings~\cite{2010smart_building}, power grids~\cite{2012smart_grid}, industrial and environmental control~\cite{2010industrial_automation}.
% 			An NCS  is controlled and monitored through a communication network, where control and estimation is performed in a distributed and/or decentralized fashion while using the underlying communication network for information exchange. 
% 			%
% 			% As a result, the paradigm of co-design is of particular importance. 
% 			%
% 			In this setting, due to potential constraints on the communication and computation resources, sensor scheduling for state estimation and actuator scheduling for controller synthesis are two important and challenging problems, and efficient algorithms are sought for solving them.
% {\blue 
Potential applications  include smart buildings~\cite{2010smart_building},
% power grids~\cite{2012smart_grid}, 
industrial and environmental control~\cite{2010industrial_automation}.
%
% An NCS  is controlled and monitored through a communication network, where control and estimation is performed in a distributed and/or decentralized fashion while using the underlying communication network for information exchange. 
%
% As a result, the paradigm of co-design is of particular importance. 
%
% In this setting, 
Due to potential constraints on the communication and computation resources of NCSs, sensor scheduling 
% for state estimation 
and actuator scheduling 
% for controller synthesis 
are two important and challenging problems, and efficient algorithms are sought for solving them.
% Therefore, efficient sensor/actuator scheduling algorithms are desired, which utilize a limited number of sensors/actuators at each time step.
% }

The majority of the existing work focuses on sensor scheduling problems and their  variants. 
Several approaches (e.g., stochastic selection \cite{Gupta2006auto}, search tree pruning \cite{Vitus2012auto}, greedy selection \cite{Tzoumas2021tac}, semidefinite programming based trajectory tracking \cite{MAITY2022110078}) have been proposed to solve such problems.
In contrast, the actuator scheduling problem has received much lesser attention. 
While sensor scheduling problems focus on minimizing (a function of) the estimation error,  the actuator scheduling directly affects the controllability and stability of the system as well as the control performance.
Therefore, a significant portion of the work on actuator scheduling focuses on studying the effects of actuator scheduling on the controllability and stability of the systems, e.g., \cite{pasqualetti2014controllability, summers2014optimal, cortesi2014submodularity, Siami2021tac, maheshwari2021stabilization} and others.
It is shown in \cite{pasqualetti2014controllability,summers2014optimal,cortesi2014submodularity} that several classes of energy related metrics associated with the controllability Gramian 
have a structural property (\textit{modularity}) that
allows for an approximation guarantee by using a simple greedy heuristic.
These problems are further investigated in \cite{Siami2021tac}, where a framework of sparse actuator schedule design was developed that guarantees performance bounds for a class of  controllability metrics.
Except \cite{maheshwari2021stabilization}, these works assume a time invariant scheduling problem, which is likely to be suboptimal and may impose restrictions on controllability for large systems.
In \cite{maheshwari2021stabilization} the authors use a round robin scheme for selecting the actuators and show that local stability is attained if the switching between the actuators is fast enough. 
The efficacy of a time-varying scheduling over a time-invariant one for interconnected systems {\blue is} also demonstrated in \cite{nozari2017time}. 
However, how to find the optimal time-varying schedule still remains unanswered.
% In \cite{Gupta2006auto}, stochastic sensor selection algorithms were developed for state estimation. Upper and lower bounds on the achievable error covariance matrices were obtained.
% %
% In \cite{Vitus2012auto}, a search tree pruning approach was proposed for obtaining suboptimal sensor schedules. 
% %
% The performance of such algorithms was investigated and error bounds were provided.
% %
% In \cite{MAITY2022110078}, a sensor scheduling problem was solved from a sensor design perspective and suboptimality guarantees were provided.
% %
% For related work, see also \cite{Tzoumas2021tac, Smith2015auto,Boyd2009tsp,junfengwu2013tac} and the references therein.
% % }

%
The efficacy of the abovementioned controllers on a system with different performance criteria (e.g., a quadratic cost function) is unknown and likely to be suboptimal since these works solely focus on the controllability and stability aspect of the system.
In contrast to those works, a few existing works \cite{shiling2013tac, 2019tii_scheduling, bommannavar2008optimal} consider a linear-quadratic optimal control problem for actuator scheduling.
However, the focus on these works is to decide whether to activate the \textit{one single} actuator available at each time or not. 
%
% Furthermore, these works also investigated the effects of the lossy communication channel on the actuator activation.

% {\blue
% While sensor scheduling problem has been extensively studied, the actuator scheduling problem has been much less explored \cite{shiling2013tac, 2019tii_scheduling,summers2014optimal, cortesi2014submodularity, Siami2021tac}.
% %
% In \cite{shiling2013tac},  for a linear system with one control actuator, control data schedules were proposed for a finite-horizon LQR control problem with limited   communication.
% %
% An event-driven method was proposed for a joint actuator scheduling and control problem in \cite{2019tii_scheduling} that meets a required  control accuracy while saving energy consumption of actuators.
% %
% In~\cite{summers2014optimal, cortesi2014submodularity}, problems of actuators 
% selection or placement that optimize real-valued
% controllability metrics were   investigated. 
% %
% It was shown that several classes of energy related metrics associated with the controllability Gramian 
% have a strong structural property, which
% allows for an approximation guarantee by using a simple greedy heuristic.
% %
% These problems were further investigated in \cite{Siami2021tac}, where a framework of sparse actuator schedule design was developed that guarantees performance bounds for a class of  controllability metrics. 
% }

Motivated by the above, in this letter  we study the  actuator scheduling problem for a finite horizon linear-quadratic control system with a number of actuators. 
We consider the case that  a nonempty subset of the actuators is active at each time.
The performance of the  actuator schedule is measured by a finite horizon quadratic cost function of the system state and control plus the cost of using each actuator (representing e.g., energy consumption).
This problem is combinatorial in nature and is NP-hard in general. 
Due to space limitations, we first restrict ourselves to the case that only one actuator is activated at each time and that all actuators have equal actuation costs.
We then provide discussions and simulation results on the general cases that multiple actuators are activated at each time and that the  actuators have different actuation costs.
% from the perspective of control design.
%

% 
% Based on this relation, we propose a suboptimal solution to the actuator scheduling problem.
%

% \ifnewversion
% {
The main contributions of this letter are the following: 
(i) We propose a convex relaxation to the actuator scheduling problem, and use its solution as a `reference' to design an algorithm for solving the original NP-hard scheduling problem. 
(ii) 
% Using dynamic programming arguments, 
We  provide a suboptimality bound for the proposed algorithm.
% {\blue  
(iii) We further show that our results can be extended to the cases with multiple actuator scheduling and actuation costs.
% }

% \else 
% {\red
% The main contributions of this letter are the following. 
% (i) We start with an LQG control problem with constraints, and analyze the connections between the actuator scheduling problem and this LQG control problem.  By investigating their relation, we propose a scheduling algorithm.
% (ii) Using dynamic programming arguments, we  provide suboptimality guarantees for the proposed scheduling algorithm.
% % {\blue  
% (iii) We further show that our results can be extended to the cases with multiple actuator scheduling and actuation cost.
% % }
% }

% \fi 

% \alert{To cite: \cite{summers2014optimal, cortesi2014submodularity}}

% {\blue  
The outline of this letter is  as follows: 
% In Section \ref{sec_pre}, we introduced  basic notations and review  LQG control. 
In Section \ref{section_prob}, we formulate the  actuator scheduling problem, which is  solved in Section \ref{section_solution}. 
% In particular, we discuss the convex relaxation method in Section~\ref{subsec_scheduling} and analyze our proposed algorithm in Section~\ref{subsec_subopt}. 
%
% {\blue
	In Section~\ref{sec_discussion}, we   provide discussions on  multiple actuator scheduling and actuation costs.
% }
%
Simulation results are provided in Section~\ref{sec_simu}. 
% to illustrate the performance of our proposed algorithm.
Section~\ref{sec_conclusion} concludes this letter.
% }

% \vspace{ 6 pt}
% \section{Preliminaries}\label{sec_pre}

% \subsubsection*{Notation} 
{\em \textbf{Notation:}}
We denote the set of real numbers and positive real numbers  by $\mathbb{R}$ and $\mathbb{R}_+$, respectively.
The  set of $n$ dimensional vectors over $\mathbb{R}$ is denoted by $\mathbb{R}^n$ and the set of real $n \times m$ matrices is denoted by $\mathbb{R}^{n\times m}$.
% by $\mathbb{R}^{n\times *} $ the set of real matrices with $n$ rows, and by $\mathcal{P}$ the set of all positive definite matrices of all dimensions.
%
The identity matrix is denoted by $I$.
For a given matrix $A$, its transpose and inverse (if exists) are denoted by $A^\top$ and $A^{-1}$, respectively.
For a symmetric matrix $P$, we denote $P \succ 0$ ($P \succeq 0$) if it is positive definite (positive semidefinite). %, and denote $P \prec 0$ ($P \preceq 0$) if it is negative definite (negative semidefinite).
The trace of a square matrix $A$ is denoted by ${\rm tr} (A)$.
The Frobenious norm of a matrix $A$ is denoted by $\| A\|_F$.
We use $\mathbb{E} [x]$ to denote the expectation of a random variable $x$.

\section{Problem Formulation}\label{section_prob}

% \ifnewversion

% {%\blue 
We consider a system with $N$ actuators of the form 
\begin{align}\label{system11}
	x_{t+1}=A_t x_t + \textstyle \sum_{j=1}^N \gamma_t(j) B_t(j) u_t(j) + w_t,   
\end{align}
where $A_t \in \mathbb{R}^{n \times n}$, $B_t(j)  \in \mathbb{R}^{n \times m_j}$,  $x_t\in \R^n$   the state,   $u_t(j) \in \R^{m_j}$ the input from the $j$-th actuator, and $w_t\in \R^n$   an independent sequence of Gaussian random variables with $w_t \sim \mathcal{N}(0,W_t)$. 
 The initial state is  $x_0 \sim \mathcal{N}(0,W_{-1})$ and it is independent of $w_t$ for all $t$. 
The matrix $B_t(j)$ describes how the control input, $u_t(j)$, of the  $j$-th  actuator enters the system at the time $t$. The parameters $\gamma_t(j)$ are binary-valued such that $\gamma_t(j) =1$ if the $j$-th actuator is activated at the time $t$ and $\gamma_t(j) =0$ otherwise.

We consider the  actuator scheduling problem that at each time only $N_t (1 \leq N_t \leq  N)$ out of the  $N$ actuators   are used  to control the system \eqref{system11} at time $t$.
In this case, $\gamma_t(j) =1$ for $N_t$ out of all $N$ actuators  at the time $t$. 

Consider  a standard  finite horizon quadratic cost function  
\begin{align} \label{eq_LQG_cost11}
	J_1 =	\mathbb{E} \Big[ &\textstyle \sum_{t=0}^{T-1} \Big(x_t^\top Q_t x_t + \sum_{j=1}^N \gamma_t(j) {u_t(j)}\!^\top R_t(j) u_t(j) \Big) \nonumber \\
	&+ x_T^\top Q_T x_T \Big],
\end{align} 
where  $Q_t, R_t(j) \succ 0$  for all $t$. In addition, consider also the actuation cost  function
$
        J_2 = \sum_{t=0}^{T-1} \sum_{j=1}^N \gamma_t(j)  c_t(j) ,
$
where   $c_t(j) \in \mathbb{R}_+$ is the cost of using actuator $j$ at time $t$. Note that $c_t(i)$ and  $c_t(j)$ are in general different for $i \neq j$ and $i,j \in \mathsf{N}\triangleq\{1,2,\ldots, N\}$, which can be  due to the fact that different  actuators may have different energy consumption or resource usage.
The objective of  the  actuator scheduling problem  is then to find   an actuator schedule   that minimizes the joint control-actuation cost  $J= J_1 +J_2$.
 
Due to space limitations, in the sequel we will restrict ourselves to the case that $N_t =1$ and $ c_t(j) =c_t$ for all~$j$. 
In other words, we consider the case that exactly one out of the $N$ actuators is used at each time  and that each actuator has the same  actuation cost.  
The assumption $c_t(j) =c_t$ leads to $J_2$ being independent of the actuator schedule and therefore, minimizing $J$ is equivalent to minimizing $J_1$.
The discussions on the cases with multiple actuators and actuation costs will be provided afterwards in Section \ref{sec_discussion}.

Since only one actuator is used at each time and the actuation cost is independent of the chosen actuator, system~\eqref{system11} can now be rewritten as 
\begin{align}\label{system}
	x_{t+1}=A_t x_t+B_t(j) u_t(j) + w_t, \  j\in \mathsf{N} ,
\end{align}
and the associated cost function becomes
\begin{equation} \label{eq_LQG_cost2}
\textstyle	J =	\mathbb{E} \Big[ x_T^\top Q_T x_T + \sum_{t=0}^{T-1} \Big(x_t^\top Q_t x_t + {u_t(j)}^\top R_t(j) u_t(j) \Big) \Big].
\end{equation} 
Let $\sigma : [0,T-1]\to \mathsf{N}$ be an actuator schedule function such that $\sigma(t)=j$ denotes that the $j$-th actuator is used at time~$t$ to control the system \eqref{system}.
The objective is to find an actuator schedule $\sigma$ that minimizes the   cost function \eqref{eq_LQG_cost2}.
We assume that perfect state measurement is available.
The information available at the controller at time $t$ is denoted by $\mathcal{I}_t$, with $\mathcal{I}_t = \mathcal{I}_{t-1}  \bigcup \{ x_t \}$ for all $t\geq 1$ and $\mathcal{I}_0 = \{ x_0 \}$.
Therefore, the (optimal) controller associated with the $j$-th actuator is 
\begin{align} \label{control}
	u_t(j)=L_t(j)  \mathbb{E} [x_t \mid \mathcal{I}_t] = L_t(j) x_t,\quad j\in \mathsf{N},
\end{align}
where  the   gain matrix $L_t(j)$ is given by
\begin{equation}\label{control_gain}
	L_t(j) = -S_t(j)^{-1} {B_t(j)}^\top K_{t+1}(\sigma) A_t,
\end{equation}
and the matrices $K_t(\sigma),  S_t(j) \succ 0$ are given recursively by
% 	\footnote{\blue Is $K_t>0$ always?}
%\begin{equation}
% {\small 
\begin{align}
    & S_t(j) := {B_t(j)}^\top K_{t+1} (\sigma) B_t(j) + R_t(j) \label{eq:St}\\
    & K_t (\sigma) = Q_t +  A_t^\top  K_{t+1} (\sigma) A_t - \nonumber\\
    &~~~A_t^\top K_{t+1} (\sigma) B_t (\sigma(t)) S_t(\sigma(t))^{-1} {B_t (\sigma(t))}^\top\!K_{t+1}(\sigma) A_t\nonumber\\
    &  K_T(\sigma) : = K_T  = Q_T.  \label{ARE_i}
\end{align}
% {\blue 
Note that, until now, we have not yet not fixed the schedule $\sigma(t)$ at time $t$, and $j$   denotes the $j$-th actuator.
% }
Notice that, for any $t \le T-1$, the matrix $K_t(\sigma)$ depends on the actuator schedule for the interval $[t, T-1]$.
Therefore, the matrix $S_t(j)$ defined in \eqref{eq:St} depends on the actuator schedule for the interval $[t+1, T-1]$, since it depends on $K_{t+1}(\sigma)$. 
Thus, the optimal gain $L_t(j)$ associated with the $j$-th actuator depends on the \textit{future} schedule for the time interval $[t+1,T-1]$.
% }
%\end{equation}
From the LQG theory \cite{book_DPOC}, the cost $J$ for a given  schedule $\sigma$ is equal to $\sum_{t=0}^{T } {\rm tr }( K_{t}(\sigma) W_{t-1})$.

Before proceeding, to maintain brevity in the subsequent analysis, we define two matrix valued functions:
\begin{subequations} \label{function_gh}
	\begin{align}
		 & g_t(j,M) :=M \nonumber \\
		&~~~-M{B_t(j)} \big({B_t(j)}\!^\top M {B_t(j)}  +R_t(j)\big)^{-1}{B_t(j)}\!^\top M, \label{function_g}\\
		&h_t(M) :=A_{t}^\top M A_{t} +Q_t. \label{function_h}
	\end{align}
\end{subequations}

By substituting \eqref{function_gh} into \eqref{ARE_i}, we obtain that 
\begin{subequations} \label{ARE_modified}
	\begin{align}
		& K_{t\mid t+1}(\sigma) :=g_{ t}({ \sigma(t)},K_{t+1}(\sigma)), \label{K_t_mid}\\
		& K_{t}(\sigma)=h_t( K_{t\mid t+1}(\sigma)),\quad K_{T}=Q_T.\label{K_t}
		%   \\
		%   &C_t(\sigma)=C^{\sigma(t)}_t, \quad \V_t(\sigma)=\V^{\sigma(t)}.
	\end{align}
\end{subequations}
{ In what follows, we will supress $\sigma (t) = \sigma$ to maintain notation brevity.}
The optimal actuator scheduling problem that we consider is then formulated as follows.
\begin{problem}[Actuator Scheduling Problem]\label{prob1}
	Given   system \eqref{system} and $N$ actuators,
% 	\footnote{\red Is it ok to say `control actuator'? I am asking because in your recent cdc paper, you write ``We assume that the controller and the actuator are collocated''. I assume that we have to say something about the controller and actuator here as well? Or no really necessary? \dm{Maybe we just say `actuator' instead of 'control actuator'. The `controller' is what prescribing the optimal control to be of the form \eqref{control} GIVEN that the $i$-th actuator is operational.}} %
	 find a  schedule $\sigma:[0,T-1] \to \mathsf{N}$ that solves the following optimization problem:
	\begin{align*}
		\min \ &{ \textstyle \sum_{t=0}^{T } {\rm tr }( K_{t}(\sigma) W_{t-1})}\\
		\textnormal{subject to} \     & K_{t\mid t+1}(\sigma)=g_{ t}({ \sigma},K_{t+1}(\sigma)), \ {Q_T \succ 0,} \\
		& K_{t}(\sigma)=h_t( K_{t\mid t+1}(\sigma)),\ K_{T}=Q_T 
	\end{align*}
	with the variables $\sigma, K_t, K_{t\mid t+1}.$
\end{problem}

% Due to the discrete mapping of the scheduling function $\sigma(\cdot)$,

% \ifnewversion

% {%\blue

% {\blue	
Problem \ref{prob1} is combinatorial in nature and NP-hard in general \cite{Tzoumas2016cdc}.
% }
%As mentioned before, many approaches have been presented in past to solve the proposed problem exactly or a relaxed version of it.
% While most of the existing work  on actuator/sensor scheduling relies on integer programming or relaxations, 
%
We now propose an efficient solution to Problem~\ref{prob1} using a convex relaxation.

\section{Actuator Scheduling with Suboptimality Guarantees} \label{section_solution}

% \ifnewversion
% {%\blue 
In this section, we solve  Problem \ref{prob1} and provide a suboptimal solution  that  is computationally inexpensive. 
We will propose a convex relaxation to the problem (see Problem~\ref{prob_control_relax}) and will use the solution of the relaxed problem as a `reference' to find a solution to Problem~\ref{prob1}.
In Section~\ref{subsec_scheduling} we will propose a \textit{tracking algorithm} that finds a solution which is `close' to the reference solution found from solving the relaxed convex optimization problem. 
The suboptimality bound of the proposed algorithm is discussed using dynamic programming type arguments in Section~\ref{subsec_subopt}.
% }
% \else 

% {\red
% In this section, we will solve  Problem \ref{prob1} and provide a suboptimal solution  that  is computationally inexpensive. 
% %
% Instead of looking for an actuator schedule $\sigma(\cdot)$ directly, in Subsection~\ref{subsec_LQG_constrains} we will first focus on an  LQG control problem with constraints. 
% %
% Subsequently, in Subsection \ref{subsec_scheduling}, we will employ the results from Subsection \ref{subsec_LQG_constrains} to obtain a suboptimal solution for  Problem~\ref{prob1}. 
% %
% The suboptimality analysis of the proposed solution is provided in Subsection~\ref{subsec_subopt}. 
% }

% \fi 

Before proceeding, we first reformulate Problem \ref{prob1} into a form that is easier for the analysis afterwards. 
According to \eqref{function_h} and \eqref{K_t}, we  have $ K_t(\sigma) = A_t\T K_{t|t+1}(\sigma)A_t + Q_t$, $t=0,1,\ldots, T-1$ and $ K_{T} =Q_T$. Subsequently, we obtain
\begin{equation}\label{eq_change}
\textstyle    \sum_{t=0}^{T } {\rm tr }( K_{t} (\sigma)W_{t-1}) = \sum_{t=0}^{T-1} {\rm tr }( K_{t|t+1}(\sigma) \bar W_{t-1}) +  r,
\end{equation}
where $r = \sum_{t=0}^{T } {\rm tr }( Q_{t} W_{t-1}) $ and $\bar{W}_{t-1} =  A_t W_{t-1}A_t\T$.
Note that $r$ is independent of  $\sigma, K_t(\sigma)$ and $K_{t\mid t+1}(\sigma).$

Next, we define two matrices $P_t$ and $P_{t|t+1}$   as follows
\begin{align*}
\textstyle	
P_t(\sigma) := K_t^{-1}(\sigma) ,\quad P_{t\mid t+1}(\sigma) :=  K^{-1}_{t\mid t+1}(\sigma).
\end{align*}
%
% It then follows  from \eqref{function_g} and \eqref{K_t_mid} that
% \begin{align*}
% \textstyle
%     P^{-1}_{t\mid t+1}(\sigma) & =  P_{t+1}^{-1}(\sigma)  -P_{t+1}^{-1}(\sigma)  B_t ({\sigma}) 
% 	\big( B_t^\top ({\sigma}) P_{t+1}^{-1}(\sigma)  B_t ({\sigma}) \\
% 	& \qquad \qquad +  R_t ({\sigma})  \big)^{-1}
% 	B_t^\top ({\sigma})P_{t+1}^{-1}(\sigma).
% \end{align*}
%where $S_t(\sigma) \triangleq C_t(\sigma)P_{t|t-1}(\sigma)C_t(\sigma)\T +\V_t(\sigma)$ and $C_t(\sigma) \triangleq C^{\sigma(t)}_t$, $\V_t(\sigma) \triangleq  \V^{\sigma(t)}$.
According to Woodbury matrix equality%\footnote{$(A+UCV)^{-1}\triangleq A^{-1}-A^{-1}U(C^{-1}+VA^{-1}U)^{-1}VA^{-1}$.}
,     we have
\begin{align}\label{P}
	P_{t\mid t+1}(\sigma) =  P_{ t+1}(\sigma) + B_t ({\sigma}) R_t^{-1} ({\sigma})  B_t^\top ({\sigma}).
\end{align}

Using \eqref{eq_change} and  the new variables $P_{t\mid t+1}(\sigma) $, $ P_{t}(\sigma) $, Problem~\ref{prob1} can be rewritten as Problem~\ref{prob2}.
\begin{problem} \label{prob2}
	Given system \eqref{system} with $N$  actuators,
% Let $\bar W_{t-1} = A_t W_{t-1}A_t\T$.  
	find a  schedule $\sigma :[0,T-1]\to \mathsf{N}$ that solves the following:
	\begin{align*}
		%   \min \quad &J =	\mathbb{E} \bigg\{ x_T^\top Q_T x_T + \sum_{t=0}^{T-1} x_t^\top \Big(Q_t + {K_t}^\top(\sigma) R_t(\sigma) {K_t} (\sigma)\Big) x_t   \bigg\}, \\
		\min \ & {\textstyle \sum_{t=0}^{T -1} {\rm tr }( K_{t|t+1} (\sigma)\bar W_{t-1})}\\
		\textnormal{subject to} \  &  K_{t\mid t+1}(\sigma) =  P^{-1}_{t\mid t+1}(\sigma), \ { Q_T \succ 0,} \\ %  K_{t} (\sigma) = h_t(\sigma, K_{t\mid t+1}(\sigma)) \\
		& P_{t\mid t+1}(\sigma) =  P_{ t+1}(\sigma) + B_t ({\sigma}) R_t^{-1} ({\sigma})  B_t^\top ({\sigma}), \\
		& {P^{-1}_t(\sigma) = h_t( K_{t\mid t+1}(\sigma))},\  P_{T}=Q_T^{-1} 
	\end{align*}
	with variables $\sigma, K_{t}, K_{t\mid t+1}, P_{t}, P_{t\mid t+1}.$
\end{problem}

Note that, although the  constraints in Problem~\ref{prob1} and Problem~\ref{prob2}   appear differently, one can in fact verify that these two problems are equivalent. 
%

% \ifnewversion

% \NewVersionText{
Let us denote $V_t(\sigma) := B_t ({\sigma}) R_t^{-1} ({\sigma})  B_t^\top ({\sigma})$ and the set $\mathsf{V}_t := \{{B_t(j)} {R_t(j)}^{-1} {B_t(j)}^\top~:~j \in \mathsf{N}\}$.
Therefore, we may rewrite the constraints in Problem~\ref{prob2} to be %
$K_{t\mid t+1}  =  P^{-1}_{t\mid t+1},\ P_{t\mid t+1} =  P_{ t+1}  + V_t, \ V_t \in \mathsf{V}_t, \  P_{t}   = \big(h_t(  K_{t\mid t+1} )\big)^{-1},$ and  $P_{T}=Q_T^{-1}.$
We have suppressed the arguments $\sigma$ in the variables to maintain notational brevity.
We can further relax the  constraints in Problem~\ref{prob2} to their equivalent matrix inequality  $K_{t\mid t+1}\succeq P^{-1}_{t\mid t+1}$, and $ P_{t}   \preceq (h_t(  K_{t\mid t+1} ))^{-1}$. 
Using Schur complement, one may write $K_{t\mid t+1}\succeq P^{-1}_{t\mid t+1}$ as the Linear matrix inequality {\small $\begin{bmatrix} K_{t\mid t+1} & I \\ I & P_{t\mid t+1} \end{bmatrix} $} $ \succeq  0$.
Similarly, using the definition of $h_t(\cdot)$ from \eqref{function_h}, the Woodbury matrix inverse identity, and Schur complement, we obtain the following problem from Problem~\ref{prob2}.
\begin{problem} \label{prob_control_relax}
Given $\mathsf{V}_t := \{{B_t(j)} {R_t(j)}^{-1} {B_t(j)}^\top ~: ~j \in \mathsf{N}\} \subset \R^{n \times n}$, solve the following optimization problem
% }
	\begin{align*}
		%   \min \quad &J =	\mathbb{E} \bigg\{ x_T^\top Q_T x_T + \sum_{t=0}^{T-1} x_t^\top \Big(Q_t + {K_t}^\top  R_t  {K_t}  \Big) x_t   \bigg\} \\
		\min \  & {\textstyle \sum_{t=0}^{T-1} {\rm tr }( K_{t|t+1} \bar W_{t-1})}\\
		\textnormal{subject to} \   & P_{t\mid t+1} =  P_{ t+1}  + V_t,\ V_t \in \mathsf{V}_t, \ P_{T}=Q_T^{-1}, \\
		&\begin{bmatrix} K_{t\mid t+1} & I \\ I & P_{t\mid t+1} \end{bmatrix}  \succeq  0,\ { Q_T \succ 0,}\\
		&\begin{bmatrix}
			Q_{t}^{-1} - P_{t} & Q_{t}^{-1} A_t^\top \\ 
			A_t Q_{t}^{-1} &  P_{t\mid t+1} + A_t Q_t^{-1} A_t^\top \end{bmatrix} \succeq 0  
	\end{align*}
	with variables $ K_{t\mid t+1}, P_{t\mid t+1},  P_{t}$, and $V_t
	$.
\end{problem}
Notice that the constraint $V_t \in \mathsf{V}_t$ is sufficient to enforce the scheduling constraint $\sigma:[0,T-1] \to \mathsf{N}$.

While Problem~\ref{prob_control_relax} is a relaxation of Problem~\ref{prob2}, we now show a key result that  an optimal solution to Problem~\ref{prob_control_relax} is also an optimal solution to Problem~\ref{prob2}.

\begin{thm} \label{thm_equivalnece}
An optimal solution of the relaxed problem (Problem \ref{prob_control_relax}) is also an optimal solution of the original problem (Problem \ref{prob2}), and vice-versa.
\end{thm}

\iflongversion

\begin{proof}
% \blue
 The proof of this theorem is along the lines of \cite[Theorem~1]{MAITY2022110078}.
	% 		\red [The proof has to be modified, since the cost is of different form now.]
	%		
	First, note that, due to the relaxations, any feasible solution of Problem~\ref{prob2} is a feasible solution for Problem~\ref{prob_control_relax}, and hence the optimal solution of Problem~\ref{prob2} is a feasible solution for Problem~\ref{prob_control_relax}.
	The theorem is proved once we show that for every feasible solution of Problem \ref{prob_control_relax} there exists a feasible solution for Problem \ref{prob2} that produces the same, if not a smaller, objective value.
	
	In order to show that, let the tuple $\{K_{t\mid t+1}, P_{t\mid t+1},  P_{t}\}$   denote a feasible solution of Problem \ref{prob_control_relax}. 
	Let us construct a  new tuple $\{\bar{K}_{t\mid t+1}, \bar{P}_{t\mid t+1},  \bar{P}_{t}\}$  as follows
	\begin{equation}\label{variables_bar}
		\begin{aligned}
			&   \bar{P}_{t\mid t+1}  =  \bar{P}_{t+1}  + \bar{V}_t ,\quad \bar{P}_{t}   = \big(h_t(  \bar{K}_{t\mid t+1} )\big)^{-1}  \\
			& \bar{V}_t =  {P}_{t\mid t+1} - {P}_{ t+1}   ,\quad   \bar{K}_{t\mid t+1}  = \bar{P}^{-1}_{t \mid t+1} ,\quad \bar{P}_{T}= P_T. 
		\end{aligned} 
	\end{equation}
	
	We will now show that 	$\bar{P}_{t \mid t+1} \succeq {P}_{t \mid t+1}$, $\bar K_{t \mid t+1}\preceq K_{t \mid t+1}$ and $\bar P_{t+1} \succeq P_{t+1}$ for all $t$.
	We start with $t=T$, it   follows from $\bar{P}_{T}= P_T$ that $\bar{P}_{T} \succeq P_T$. Since $\bar{P}_{t\mid t+1}  =  \bar{P}_{t+1}  + \bar{V}_t$ and $ \bar{V}_t =  {P}_{t\mid t+1} - {P}_{ t+1} $, it follows that $\bar{P}_{T-1\mid T} - {P}_{T-1\mid T}   =  \bar{P}_{T}   - {P}_{T}$. 
		Since $\bar{P}_{T} \succeq P_T$, this implies that $ \bar{P}_{T-1\mid T} \succeq {P}_{T-1\mid T}.$ 
		Note that $\bar{P}_{T-1\mid T} \succeq {P}_{T-1\mid T} \succ0$, it then follows that $\bar{P}_{T-1\mid T}^{-1} \preceq {P}_{T-1\mid T}^{-1}$. 
		Now, recall that  $K_{t\mid t+1}\succeq P^{-1}_{t\mid t+1}$, and  that $\bar{K}_{t\mid t+1}  = \bar{P}^{-1}_{t \mid t+1} $ in \eqref{variables_bar}. It then follows that
		\begin{align*}
		    \bar{K}_{T-1\mid T}  = \bar{P}^{-1}_{T-1 \mid T} \preceq {P}_{T-1\mid T}^{-1} \preceq {K}_{T-1\mid T} ,
		\end{align*}
		that is $\bar{K}_{T-1\mid T}\preceq {K}_{T-1\mid T}.$
Next, recall   that 
		\begin{align*}
		 &  \bar{P}_{t} ^{-1} =  h_t(\bar{K}_{t\mid t+1} ) :=A_{t}^\top \bar{K}_{t\mid t+1} A_{t} +Q_t ,\\
		  &	    {P}_{t} ^{-1} =  h_t({K}_{t\mid t+1} ) :=A_{t}^\top {K}_{t\mid t+1} A_{t} +Q_t.
		\end{align*}	
		It then follows from $\bar{K}_{T-1\mid T}\preceq {K}_{T-1\mid T}$ that $\bar P_{T-1}^{-1} \preceq P_{T-1}^{-1} $. This implies that $ \bar P_{T-1}  \succeq P_{T-1} .$  
		By induction, the same procedure applies for all $t = T-1, T-2,\ldots, 0$.
	Thus,
	$\bar{P}_{t \mid t+1} \succeq {P}_{t \mid t+1}$, $\bar K_{t \mid t+1}\preceq K_{t \mid t+1}$ and $\bar P_{t+1} \succeq P_{t+1}$ for all~$t$.
	
	Next, let   $ K_{t\mid t+1}, P_{t\mid t+1},  P_{t}$, and $V_t$ be solutions of Problem 3. Then $ K_{t\mid t+1}, P_{t\mid t+1},  P_{t}$, and $V_t$ satisfy the constraints of Problem 3, in particular, $P_{t\mid t+1} =  P_{ t+1}  + V_t,\ V_t \in \mathsf{V}_t$. This implies that 
$  V_t = P_{t\mid t+1} -  P_{ t+1}  ,\ V_t \in \mathsf{V}_t.$ 
Now, recall that in \eqref{variables_bar} we have chosen  $\bar{V}_t =  {P}_{t\mid t+1} - {P}_{ t+1} $. It then follows from   $\bar{V}_t =  {P}_{t\mid t+1} - {P}_{ t+1} $ and $ V_t = P_{t\mid t+1} -  P_{ t+1}  ,\ V_t \in \mathsf{V}_t $ that   $\bar V_t \in  {\sf V}_t$.

	%%Moreover, this construction of $\{\bar P_t,\bar Q_t,\bar Q_{t|t-1}\}$ also ensures that
	%% \begin{align*}
	%%     &\bar P_t=\bar Q_t^{-1},\\
	%%     &\bar Q_{t}=\bar Q_{t|t-1}+R_t,\\
	%%     &\bar Q_{t|t-1}=\left(h_t(\bar P_{t-1})\right)^{-1}.
	%% \end{align*}
% 	Note also that the matrix $\bar V_t$ in \eqref{variables_bar}  satisfies $\bar V_t\in \mathsf V_t$.
	%
Now, since the tuple $\{\bar{K}_{t\mid t+1}, \bar{P}_{t\mid t+1} ,  \bar{P}_{t}\}$ satisfies all the constraints of Problem~\ref{prob2}, this implies that it is a feasible solution of Problem~\ref{prob2}. 
	%
% 	Furthermore, 
% 	it follows from $\bar{P}_{t \mid t+1} \succeq {P}_{t \mid t+1}$ that $\bar{K}_{t \mid t+1} \preceq {K}_{t \mid t+1}$ for all $t$. 
% 	Subsequently, we have that $\bar W_{t-1}^{\frac{1}{2}} \bar{K}_{t \mid t+1} \bar W_{t-1}^{\frac{1}{2}} \preceq \bar W_{t-1}^{\frac{1}{2}} {K}_{t \mid t+1} \bar W_{t-1}^{\frac{1}{2}}$, 
	%
% 	and that $\sum_{t=0}^{T-1}\tr(\bar W_{t-1}^{\frac{1}{2}} \bar{K}_{t \mid t+1} \bar W_{t-1}^{\frac{1}{2}}) \preceq \sum_{t=0}^{T-1}\tr(\bar W_{t-1}^{\frac{1}{2}} {K}_{t \mid t+1} \bar W_{t-1}^{\frac{1}{2}}) $. 
	%
Next, note that  $ \bar{K}_{t \mid t+1} \preceq {K}_{t \mid t+1}$ for all $t$,	this then implies that 
	$\sum_{t=0}^{T-1}\tr(\bar{K}_{t \mid t+1} \bar W_{t-1}) \preceq \sum_{t=0}^{T-1}\tr(K_{t \mid t+1} \bar W_{t-1})$.
	Therefore, for any feasible solution of Problem~\ref{prob_control_relax} we can construct a feasible solution for Problem~\ref{prob2} that produces the same, if not less, cost.
	This completes the proof.  
	% 		Proof to be provided.
\end{proof}

\fi

%%%%%%%%%%%%%%%%%
% proof
%%%%%%%%%%%%%%%%%

% \begin{remark}
% % \orange
% Theorem \ref{thm_equivalnece} shows that the LMI-based relaxations introduced in Problem \ref{prob_control_relax} do not affect the optimality since an optimal solution of the relaxed problem is also optimal for the original problem.
% %
% This is a key advantage of this approach, as the LMI-based relaxations retain the optimality.
% Moreover, since Problem \ref{prob_control_relax} is a mixed integer semidefinite program, it can be directly solved using efficient numerical techniques~\cite{gally2018framework}. 

% \end{remark}
% \end{comment}

% \else % this else for longversion

% {\blue A proof of Theorem \ref{thm_equivalnece} can be found in \cite{jiao2022}.}

% \begin{remark}
% %  \orange
% Theorem \ref{thm_equivalnece} shows that the LMI-based relaxations introduced in Problem \ref{prob_control_relax} do not affect the optimality, since an optimal solution to the relaxed problem is also optimal for the original problem.
% This is a key advantage of this approach, as the LMI-based relaxations retain the optimality.
% Moreover, since Problem \ref{prob_control_relax} is a mixed integer semidefinite program, one may attempt to directly solve it using available numerical techniques~\cite{gally2018framework}. 
% \end{remark}

% \fi % this fi for ending longversion

\begin{remark}
%  \orange
Theorem \ref{thm_equivalnece} shows that the LMI-based relaxations introduced in Problem \ref{prob_control_relax} do not affect the optimality, since an optimal solution to the relaxed problem is also optimal for the original problem.
This is a key advantage of this approach, as the LMI-based relaxations retain the optimality.
Moreover, since Problem \ref{prob_control_relax} is a mixed integer semidefinite program, one may attempt to directly solve it using available numerical techniques~\cite{gally2018framework}. 
\end{remark}

Next, note that Problem~\ref{prob_control_relax} is convex if $\mathsf{V}_t$ is a convex set for all $t$. 
When $\mathsf{V}_t$ is not convex, one could take the convex hull of the set $\mathsf{V}_t$ to make Problem~\ref{prob_control_relax} convex.
In our case, since $\mathsf{V}_t$ is a collection of $N$ matrices $\{V_t(1),\ldots,V_t(N)\}$ where $V_t(j) = {B_t(j)} {R_t(j)}^{-1} {B_t(j)}^\top$ for all $t$, we replace the constraint $V_t \in \mathsf{V}_t$ with the constraints $V_t=\sum_{i=1}^N \theta^i_t V_t(i)$, $\theta^i_t\in [0,1]$ and $\sum_{i=1}^N \theta^i_t=1$.
In this case, Problem~\ref{prob_control_relax} can be further simplified to Problem~\ref{prob_SDP}.
\begin{problem}\label{prob_SDP}
	\begin{align*}
		%   \min \quad &J =	\mathbb{E} \bigg\{ x_T^\top Q_T x_T + \sum_{t=0}^{T-1} x_t^\top \Big(Q_t + {K_t}^\top  R_t  {K_t}  \Big) x_t   \bigg\} \\
		\min \ & { \textstyle \sum_{t=0}^{T-1} {\rm tr }( K_{t|t+1} \bar W_{t-1})}\\
		\textnormal{subject to}\  & P_{t\mid t+1} =  P_{ t+1}  +   \textstyle \sum_{i=1}^N \theta^i_t V_t(i), \  P_{T}=Q_T^{-1},\\
&  \textstyle \sum_{i=1}^N \theta^i_t =1, 0 \le \theta^i_t \le 1, \begin{bmatrix} K_{t\mid t+1} & I \\ I & P_{t\mid t+1} \end{bmatrix}  \succeq  0,\\
&\begin{bmatrix}
	Q_{t}^{-1} - P_{t} & Q_{t}^{-1} A_t^\top \\ 
	A_t Q_{t}^{-1} &  P_{t\mid t+1} + A_t Q_t^{-1} A_t^\top \end{bmatrix} \succeq 0, {Q_T \succ 0}  \nonumber
%		& \sum_{i=1}^\ell \theta^i_t =1, \  0 \le \theta^i_t \le 1 .
	\end{align*}
	\rev{with variables $\theta_t^i, K_{t\mid t+1}, P_{t \mid t+1},   P_{t}$.}
\end{problem}
% \begin{remark}
% \revv{For the case when $K_t$ sensors are to be scheduled at time $t$, then the constraint $\sum_{i=1}^\ell \theta^i_t = 1$ will be changed to $\sum_{i=1}^\ell \theta^i_t = K_t$.}
% \end{remark}
% In Problem \ref{P:design_SDP}, by using the fact that $\mathsf R_t$ is the convex hull of $\{R^1_t,\ldots,R^\ell_t\}$, we have replaced the constraint $R_t \in \mathsf{R}_t$ with the constraints $R_t=\sum_{i=1}^\ell \theta^i_t R^i_t$,~~ $\theta^i_t\in [0,1]$ and $\sum_{i=1}^\ell \theta^i_t=1$.
% Therefore,  when $R_t$ is confined in a given convex hull, a sensor design problem can easily be solved by using the SDP formulation presented in Problem \ref{P:design_SDP}.
%It is noteworthy at this point that a similar treatment directly on Problem~\ref{P:prob1} will not result into a convex formulation as we derived by using sensor-design framework. 

At this point we have a convex optimization problem (semidefinite program) in Problem~\ref{prob_SDP} which is much easier to solve compared to the mixed integer semidefinite program in Problem~\ref{prob_control_relax}.
If the optimal $\theta_t^j$ is binary-valued then the optimal schedule to Problem~\ref{prob1} is found by setting $\sigma(t) = j$ such that $\theta^j_t =1$. 
However, in general the optimal $\theta^j_t$ are not binary-valued and we need to design an algorithm to find a schedule $\sigma$ from the solution to Problem~\ref{prob_SDP}.

\begin{remark}\label{rem_actuator_selection}
% \alert{Why don't we just round up the maximum $\theta^i_t$ for scheduling? (The reason is it is not a good approach). See simulation}
{ At first glance, it may seem that selecting the actuator with the maximum value of $\theta^i_t$ at each time will lead to the smallest value of  ${\rm tr }( K_{t|t+1} \bar W_{t-1})$. 
However, it is not necessarily the case (see simulation   in Section~\ref{sec_simu}).
In the next section we propose a more  efficient algorithm and discuss its suboptimality bound.
%
% Furthermore, we provide a suboptimality guarantee of our algorithm. using an approximate dynamic programming type argument.
% We  will also  show this later in our simulation example.
}
\end{remark}

\subsection{Actuator Scheduling Algorithm}\label{subsec_scheduling}

% \ifnewversion

% % \else
% In this subsection, we will address the  actuator scheduling problem  formulated in Problem \ref{prob1}.
% %
% As mentioned in Remark~\ref{rem_design_schedule}, the actuator scheduling problem can be viewed as an LQG control   problem with constraints if we restrict the matrices $(B_t,R_t)$ to be one of the elements in the set $\{({B}^i_t,R_t^i)\}_{i\in \mathsf{N}}$. 
% %
% Equivalently, if we restrict $\mathsf V_t = \{V^i_t\}_{i \in \mathsf{N}}$  in Problem~\ref{prob_control_relax}, where $V^i_t={B^i_t} {R^i_t}^{-1} {B^i_t}^\top$, then we recover a solution for Problem~\ref{prob1}.
% However, solving Problem~\ref{prob_control_relax} with the non-convex constraint $V_t\in \{V^i_t\}_{i \in \mathsf{N}}$ is computationally expensive despite the availability of efficient numerical techniques  \cite{gally2018framework}. 
% Therefore, we relax the constraint $V_t\in \{V^i_t\}_{i \in \mathsf{N}}$ as $V_t\in {\rm co}\left( \{V^i_t\}_{i \in \mathsf{N}}\right)$ where ${\rm co}(\cdot)$ denotes the convex hull operation.
% % \footnote{For any set $S$ with its cardinality denoted as $|S|$, the convex hull of the set $S$ is the set $co(S)=\{x~|~ x=\sum_{i=1}^{|S|} \theta_i s_i, ~\theta_i\in [0,1],~ \sum_{i=1}^{|S|} \theta_i=1,~ s_i\in S \}$}.
% With this convex hull relaxation, the relaxed actuator scheduling problem becomes exactly  Problem \ref{prob_SDP}.

% \fi

By solving the convex relaxation in Problem~\ref{prob_SDP}, we obtain $\{\{{\theta^i_t}^o\}_{i\in \mathsf{N}}\}_{t=0}^{T-1}$, or equivalently $V^o_t=\sum_{i=1}^N{\theta^i_t}^o V^i_t$ and the associated $K^o_{t\mid t+1}$, $P^o_{t \mid t+1}$ and   $P^o_{t}$.
%
% If, ${\theta^i_t}^o\in \{0,1\}$ for all $i$ and $t$, then this relaxed optimal solution $\{\{{\theta^i_t}^o\}_{i\in \mathsf{N}}\}_{t=0}^{T-1}$ is an optimal schedule for Problem~\ref{prob1}.
% However, in general, the obtained ${\theta^i_t}^o$ are not binary-valued, and hence the solution of Problem~\ref{prob_SDP} may not be readily  used as a solution for Problem \ref{prob1}. 
In this section, we propose an algorithm that uses this solution of Problem \ref{prob_SDP} as a reference to obtain a suboptimal solution for Problem \ref{prob1}. The corresponding algorithm is presented in Algorithm \ref{algm_1}.  Note that this algorithm depends linearly on the number of the actuators.

Algorithm~\ref{algm_1} takes the solution $\{K^o_{t \mid t+1}\}_{t=0}^{T-1}$ obtained from solving Problem~\ref{prob_SDP} as an initial guess, and initializes the terminal condition $K_{T}$ at $Q_T$.
%
% The notation $\|\cdot\|_F $ in Algorithm~\ref{algm_1} represents the Frobenius norm.
The algorithm produces a  trajectory $\{K_{t \mid t+1}\}_{t=0}^{T-1}$ that is \textit{close} to the reference trajectory $\{K^o_{t \mid t+1}\}_{t=0}^{T-1}$ in Frobenius norm. %
{ 
The reasoning behind the construction of Algorithm~\ref{algm_1} is to keep the matrices $ K_{t \mid t+1}(\sigma)$ close to $ K_{t \mid t+1}^o $, and  subsequently, to keep $ \sum_{t=0}^{T-1}\tr(K_{t \mid t+1}(\sigma) \bar W_{t-1} ) $   close to $ \sum_{t=0}^{T-1}\tr( K_{t \mid t+1}^o \bar W_{t-1} )$, since $ \sum_{t=0}^{T-1}\tr( K_{t \mid t+1}^o \bar W_{t-1}) $ is the lowest  one that could possibly be achieved given the set of actuators.
The algorithm can be regarded as a \textit{trajectory-tracking} problem in the space of positive definite matrices where $ \{K_{t \mid t+1}^o\}_{t \ge 0} $ serves as the reference trajectory. 
}
% https://math.stackexchange.com/questions/1747373/minimize-mboxtraceax-over-x-with-a-positive-semidefinite-x

% The algorithm is computationally inexpensive where the schedule $\sigma(t)$ is found by finding the minimum amongst $N$ elements.

\begin{algorithm}[t]
\begin{algorithmic}[1]
\State	\textbf{Input}  $\{K^o_{t \mid t+1}\}_{t=0}^{T-1}$, $K_{T}=Q_T$
\For{$ t=T-1:0$}
\State	$M_{t}(i)\leftarrow g_t(i,K_{ t+1}),\quad i\in \mathsf{N}$
% 	$ ~~~~~~\sigma(t)\leftarrow \arg\!\min_i \| K^o_{t \mid t+1}   \bar  W_{t-1} -M_t(i)  \bar  W_{t-1}\|_F$\\
\State	$\sigma(t)\leftarrow \arg\!\min_i \| K^o_{t \mid t+1}    -M_t(i) \|_F$
\State	$K_{t \mid t+1}\leftarrow g_{t}(\sigma(t),K_{t+1})$
\State	$K_{t}\leftarrow h_t( K_{t \mid t+1})$
\EndFor
%
	% 		\textbf{for} $\blue t=0$\\
	% 		$~~~~~~ K_{0 \mid 1}(i)\leftarrow g_t(i,K_{1}),\quad i\in \mathsf{N}$,\\
	% 		$~~~~~~ M_{0}(i)\leftarrow h_0(i,K_{0 \mid 1}),\quad i\in \mathsf{N}$,\\
	% 		$\blue ~~~~~~\sigma(0)\leftarrow \arg\!\min_i \|(K^o_0  -M_0(i)) \Sigma_0\|_F$\\
	% 		$~~~~~~K_{0 \mid 1}\leftarrow g_0(\sigma(0),K_{1}),$\\
	% 		$~~~~~~K_{0}\leftarrow h_0( K_{0 \mid 1}),$\\
	% 		\textbf{end}\\
\State	\textbf{Output} $ \sigma$
	\caption{Reference Tracking Algorithm} \label{algm_1}
\end{algorithmic}
\end{algorithm}

\iflongversion

In the next subsection,  we  will provide the analysis on the suboptimality guarantee of Algorithm \ref{algm_1}.

\subsection{Dynamic Programming and Suboptimality Guarantees}\label{subsec_subopt}
We denote the associated value function  of  Problem~\ref{prob2} as
{%\small
\begin{align} \label{E:value}
\textstyle	U_t(K)=\min_{    \{\sigma(k)\}_{k=0}^{t}}    \sum_{k=0}^{t} {\rm tr}( K_{k \mid k+1}(\sigma) \bar W_{k-1}),
\end{align}
}% 
given $K_{t+1}=K$ for some $K\succ 0$.
Likewise, we denote the value function associated with Problem \ref{prob_SDP}, which is the SDP relaxation of Problem~\ref{prob1}, to be
%
% the SDP relaxation of Problem~\ref{prob1_2} (equivalent to Problem \ref{prob_SDP} with $\ell=N$ and $V_t=B_t R_t^{-1} B_t^\top$) by $U^o_t$: 
{%\small
\begin{align}
U_t^o(K)=\min_{\{\{\theta^i_k\}_{i\in \mathsf{N}}\}_{k=0}^{t}}   \textstyle \sum_{k=0}^{t} {\rm tr}( K_{k \mid k+1}(\theta) \bar W_{k-1}).
\end{align}
}% 
The difference between $U_t$ and $U^o_t$ is that the feasible choice of an actuator at time $t$ for $U_t$ has to be one of the $\{V^i_t\}_{i\in \mathsf{N}}$ (or equivalently $\{B^i_t,R^i_t\}_{i\in \mathsf{N}}$), while that for $U^o_t$ is any of the actuators that lie within the convex hull of $\{V^i_t\}_{i\in \mathsf{N}}$.
Therefore,  $U^o_t(K) \le U_t(K)$ for all $K\succ 0$. 
% In what follows, we will suppress the $K_{t}=K$ constraint in the definitions of the value function to maintain notational brevity.

Using dynamic programming, we can write
% {\small 
	\begin{align*}
		U_t(K)=&\min_{\sigma(t)}\left(\tr( K_{t \mid t+1}(\sigma) \bar W_{t-1}) +U_{t-1}(K_t(\sigma))\right)\\
		=&\min_{\sigma(t)}\Big( \tr\big( g_t(\sigma, K) \bar W_{t-1} \big) +U_{t-1}\Big(h_{ t}\big(g_{t} (\sigma, K ) \big) \Big),\\  % K = K_{t\mid t+1} 
		%    =&\min_{\sigma(t)}\left(\tr\big(g_t(\sigma(t),P)\big)+U_{t+1}\Big( h_{t+1}\big(g_t(\sigma(t),P)\big)\Big)\right),\\
		U_0(K)=& \min_{\sigma(0)} \tr\big(g_0(\sigma, K) \bar W_{-1}\big).
\end{align*}
% }% 
% An optimal schedule, although can be found by solving for the value function $U_t(\cdot)$, however that approach is equally difficult as solving Problem \ref{P:prob1}.
In the following, by exploiting   properties of  $U_t$ and the solutions ($ K^o_{t | t+1}, P^o_{t | t+1}$ and $P^o_t$) obtained from Problem~\ref{prob_SDP}, we   provide an approximate value function associated with~\eqref{E:value}.

% 	{\red  To be continued...}

% 	{\red  To be continued...}

% 	{\red  To be continued...}

%%%================================================================
% \begin{comment}

Before proceeding, let us present some useful properties of the map $g_t(\cdot,\cdot)$  defined in \eqref{function_g}, which are essential in the subsequent analysis. 
Using Lemma~1-e from \cite{sinopoli2004kalman}, one can prove that, for any fixed $i\in \mathsf{N}$,  $g_t(i,M)$ is concave in $M$.
Furthermore, we can characterize the derivative of the function $g_i(i,\cdot)$ by the following lemma.
\begin{lem}[\cite{Vitus2012auto}] \label{lem_g_concave}
% 	\blue
	For each $i\in \mathsf{N}$ and for any positive semi-definite matrices $M,X$, it follows that
	\begin{align}
		\frac{dg_t(i,M+\epsilon X)}{d\epsilon}\Big|_{\epsilon=0}= H_t(i,M) X H_t^\top(i,M),
	\end{align}
	where $	 H_t(i,M)=I-M{B_t(i)} \big({B_t(i)}\T M B_t(i) +R_t(i)\big)^{-1}{B_t(i)}\T$.
\end{lem}

% 	\begin{proof}  
% 	\red
% 		Let us define $\tilde{g}_t(i,M)=({B_t^i}\T M B_t^i +R^i)^{-1}$, and therefore,
% 		\begin{align*}
% 			\frac{d \tilde{g}_t(i,M+\epsilon X)}{d\epsilon}=-\tilde{g}_t(i,M+\epsilon X){B^i_t}\T L {B_t^i} \tilde{g}_t(i,M+\epsilon X).
% 		\end{align*}
% 		Using \eqref{function_g} and after some simplifications, we obtain
% 		\begin{align*}
% 			\frac{d {g}_t(i,M+\epsilon X)}{d\epsilon}\Big|_{\epsilon=0}=H_t(i,M) X H_t(i,M)\T. 
% 		\end{align*}      \hfill $\qed$
% 	\end{proof}
% \end{comment}

%%%=========================================================
The following proposition shows that  $U^o_t$ is locally Lipschitz, which is   important for analyzing Algorithm~\ref{algm_1}.

\begin{prop}\label{prop_Lipsc}
	For any two symmetric matrices $M_1\succ 0$ and $M_2\succ 0$ with bounded Frobenius norms, and for all $t=0,1,\ldots, T$, there exists a constant $c>0$  such that
	\begin{align}
			|U^o_t(M_1)-U^o_t(M_2)|\le c\ \| M_1-M_2  \|_F.
	\end{align}
\end{prop}
% 	The proof of Proposition \ref{prop_Lipsc} is similar to that in \cite{MAITY2022110078} and is hence omitted here.
%

%%%%%%%%%%%%%%%%%
% proof
%%%%%%%%%%%%%%%%%
% \begin{comment}

\iflongversion

\begin{proof}
	We prove this in an inductive way. 
	Let us first consider the case that $t=0$,  we then have 
	\begin{align*}
	&	U_0^o(M_1)-U_0^o(M_2) \\
	=&\min_{\{\theta^i\}_{i\in \mathsf{N}}}\! \tr\big(g_0(\theta,M_1)\bar W_{-1}\big)-\!\!\min_{\{\theta^i\}_{i\in \mathsf{N}}}\!\tr\big(g_0(\theta,M_2)\bar W_{-1}\big)\\
		\overset{(a)}{\le} & \tr\Big( \big(g_0(\theta^*,M_1)-g_0(\theta^*,M_2) \big) \bar W_{-1} \Big)\\
		\overset{(b)}{\le}& \tr\Big( \big(H_0(\theta^*,M_2)(M_1-M_2)H_0\T(\theta^*,M_2)\big) \bar W_{-1}\Big)\\
		% =&\left\langle P-Q, H_T\T(\theta^*,Q)H_T(\theta^*,Q)\right\rangle_F\\
		\le &\|M_1-M_2\|_F\|H_0\T(\theta^*,M_2) \bar W_{-1} H_0(\theta^*,M_2)\|_F 
	\end{align*}
	where $\theta^*=[\theta^{1*},\ldots,\theta^{N*}]$ in $(a)$ is a minimizer of $\tr(g_0(\theta,M_2)\bar W_{-1})$, and (b) follows from the concavity property of the function $g_0(\theta,\cdot)$ along with Lemma~\ref{lem_g_concave}. 
	From the expression of $H_0(\theta^*,M_2)$ in Lemma~\ref{lem_g_concave}, along with the fact that $M_2$ has a bounded Frobenius norm, one can verify that there exists a finite value $c>0$ such that $ \|H_0\T(\theta^*,M_2) \bar W_{-1} H_0(\theta^*,M_2)\|_F \le c$. 
	Therefore, 
% 	\begin{align*}
	$|	U_0^o(M_1)-U_0^o(M_2) | \le c\ \| M_1-M_2 \|_F.$
% 	\end{align*}
 The inductive hypothesis can be proven in a similar way.
\end{proof}

\else

For a proof of Proposition \ref{prop_Lipsc}, we refer to \cite{jiao2022}.

\fi
The following proposition states that an upper bound on $U_t$ is found from $U_t^o$. 

\begin{prop} \label{Pr:Ualpha}
	For any time $t$ and $K\succ 0$ with bounded Frobenius norm, there exists a finite $\alpha_t>0$ such that
	\begin{align*}
		U_t(K) \le U^o_t(K)+\alpha_t.
	\end{align*}
\end{prop}

Based on these propositions, we are now ready to perform an approximate dynamic programming using the value function $U_t(K)$ to design a suboptimal solution as follows.
Recall that the value function $U_t(K)$ satisfies
\begin{align*}
% 	\blue
	% 		U_t(K)=&\min_{\sigma(t)}\left(\tr(K_t(\sigma)W_{t-1})+U_{t-1}(K_{t-1|t}(\sigma))\right),
	U_t(K)=\min_{\sigma(t)}\left(\tr( K_{t \mid t+1}(\sigma) \bar W_{t-1}) +U_{t-1}(K (\sigma))\right),
\end{align*}
 it then follows from Proposition \ref{Pr:Ualpha} that  
% {\small
\begin{align*}
	U_t(K) &\le \alpha_{t-1}+\min_{\sigma(t)}\left(\tr( K_{t \mid t+1}(\sigma) \bar W_{t-1})+U^o_{t-1}(K_{t}(\sigma))\right), \\
	&\le \alpha_{t-1} +U^o_t(K^o_{t+1})  +\min_{\sigma(t)}\big(\tr( K_{t \mid t+1}(\sigma) \bar W_{t-1}) \\
	&\qquad \qquad +U^o_{t-1}(K_{t}(\sigma))-U^o_t(K^o_{t+1}) \big),
\end{align*}
% }% 
where $K^o_{t}$ is obtained from the convex relaxation in Problem~\ref{prob_SDP}. 
More specifically, by solving that relaxed problem, we obtain  $\{P^o_t,K^o_{t|t+1},P^o_{t|t+1}\}_{t=0}^{T-1}$ and, subsequently, we can construct $K^o_t=({P^{o}_{t}})^{-1}$.
				Thus, we have that $U^o_t(K^o_{t+1})=\tr(K^o_{t \mid t+1} \bar W_{t-1})+U^o_{t-1}(K^o_{t})$, and therefore,
				% {\small
					\begin{align*}  %\label{eq:upper_bound}
						% \blue
						&U_t(K)\\
						& \le \alpha_{t-1} +U^o_t(K^o_{t+1})  +\min_{\sigma(t)}\big(\tr\big((K_{t \mid t+1}(\sigma)-K^o_{t \mid t+1})\bar W_{t-1}\big) \\ 
						&\qquad \qquad	+U^o_{t-1}(K_{t}(\sigma))-U^o_{t-1}(K^o_{t})\big)\\
						% 		\end{align*}}
				% {	\begin{align*}
						&\overset{(a)}{\le} \alpha_{t-1} +U^o_t(K^o_{t+1})   +\min_{\sigma(t)}\big(\tr((K_{t \mid t+1}(\sigma)-K^o_{t \mid t+1})\bar W_{t-1})\\
						&\qquad \qquad + c { \|K_{t}(\sigma)-K^o_{t} \|_F }\big) \\ 
						% 	&\overset{(b)}{\le} \alpha +U^o_t(K^o_{t+1}) +\min_{\sigma(t)}\big(\tr((K_{t \mid t+1}(\sigma)-K^o_{t \mid t+1})\bar W_{t-1}) \\
						% 	& \qquad \qquad + { c { \|A_{t } \|_F^2}   \|K_{t  \mid t+1}(\sigma)-K^o_{t  \mid t+1}\|_F}\big) \\ 
						% 		&\le \alpha +U^o_t(K^o_{t|t+1}) \nonumber \\ 
						% 		&\qquad   +  \sqrt{n} \| \bar W_{t-1} \|_F
						% 		\| K_{t \mid t+1}(\sigma)-K^o_{t \mid t+1}  \|_F
						% 		+  {\ c \|A_{t } \|^2_F  \|K_{t  \mid t+1}(\sigma)-K^o_{t  \mid t+1}\|_F}    \\
						% 	&\blue \overset{(b)}{\le} \alpha +U^o_t(K^o_{t+1}) +\min_{\sigma(t)}\Big( \sqrt{n} \|(K_{t \mid t+1}(\sigma)-K^o_{t \mid t+1})\bar W_{t-1}\|_F \\
						% 	& \qquad \qquad + { c  \|A_{t } \|^2_F   \|K_{t  \mid t+1}(\sigma)-K^o_{t  \mid t+1}\|_F}\Big) \\ 
						% &  \overset{(c)}{\le}  \alpha +U^o_t(K^o_{t+1})    +  c_1  \min_{\sigma(t)}  \|K_{t  \mid t+1}(\sigma)-K^o_{t  \mid t+1}\|_F   ,
						% 		&\le \alpha +U^o_t(K^o_{t|t+1}) + {\red c_2\min_{\sigma}\|K_t(\sigma)-K^o_t\|_F}.
					\end{align*}
					% }% 
				% 
				where we have used Proposition 5 in (a).
				Next, note the fact that,  for two matrices $X,Y \in \mathbb{R}^{n \times n}$, we have ${\rm tr} (XY) \leq \|X\|_F \| Y \|_F$, and recall that $ K_t (\sigma)= A_t\T K_{t|t+1}(\sigma)A_t + Q_t$, $ K_t^o = A_t\T K_{t|t+1}^o A_t + Q_t$, we have 
				\begin{align*}  %\label{eq:upper_bound}
					% \blue
					&U_t(K)\\
					&  {\le} \alpha_{t-1} +U^o_t(K^o_{t+1}) +\min_{\sigma(t)}\Big(  \|(K_{t \mid t+1}(\sigma)-K^o_{t \mid t+1})\|_F \|\bar W_{t-1}\|_F \\
					& \qquad \qquad + { c      \|A_t\T (K_{t|t+1}(\sigma) - K_{t|t+1}^o) A_t\|_F}\Big) \\ 
					&  \overset{(b)}{\le}  \alpha_{t-1} +U^o_t(K^o_{t+1})    +  c_1  \min_{\sigma(t)}  \|K_{t  \mid t+1}(\sigma)-K^o_{t  \mid t+1}\|_F   ,
					% 		&\le \alpha +U^o_t(K^o_{t|t+1}) + {\red c_2\min_{\sigma}\|K_t(\sigma)-K^o_t\|_F}.
				\end{align*}
				where we have used  $\|A_t\T (K_{t|t+1}(\sigma) - K_{t|t+1}^o) A_t\|_F =  \| A_t\|_F^2 \|K_{t  \mid t+1}(\sigma)-K^o_{t  \mid t+1}\|_F $ in (b), and $c_1 :=  \| \bar W_{t-1} \|_F  + c  \|A_{t } \|_F^2$.
% }

Thus, optimizing $ \min_{\sigma}  \|K_{t  \mid t+1}(\sigma)-K^o_{t  \mid t+1}\|_F   $  in Algorithm \ref{algm_1} in fact minimizes an upper bound of the value function $U_t$, or equivalently, an upper bound of $\sum_{t=0}^{T-1} {\rm tr}( K_{t|t+1} \bar W_{t-1})$.
Therefore, in essence, Algorithm~\ref{algm_1}  performs an approximate dynamic programming type optimization by minimizing an upper bound of $U_t$.

\else

\fi

The following theorem provides a suboptimality bound of Algorithm~\ref{algm_1}. 
% The detailed proof of this lemma is omitted due to page limitation. 

\begin{thm}\label{thm_main}
	Let $\sigma,\sigma^*$ and $\theta^*$ denote the schedule obtained from Algorithm~\ref{algm_1}, the true optimal schedule of Problem~\ref{prob1}, and the solution to Problem~\ref{prob_SDP}, respectively. Then, we have
	\begin{align} \label{suboptimal_bound} 
\textstyle		\sum_{t=0}^{T } {\rm tr}( K_{t}(\sigma) W_{t-1}) \le   \sum_{t=0}^{T } {\rm tr}( K_{t }(\sigma^*) W_{t-1}) +  \epsilon,
	\end{align}
	where 
% $	\epsilon \triangleq     \| \bar W_{t-1}\|_F \Big( \textstyle \sum_{t=0}^{T-1}\frac{\lambda^{t+1}-1}{\lambda-1}\beta_{t}  +  \textstyle \sum_{t=0}^{T-1} \|K_{t \mid t+1}(\theta^*) - K_{t \mid t+1}(\sigma^*)\|_F \Big)$
% {\small 
	\begin{equation}\label{eq_epsilon}
 	\begin{aligned}
  	& 	\epsilon \triangleq     \| \bar W_{t-1}\|_F \Big( \textstyle \sum_{t=0}^{T-1}\frac{\lambda^{t+1}-1}{\lambda-1}\beta_{t} \\
	&\qquad +  \textstyle \sum_{t=0}^{T-1} \|K_{t \mid t+1}(\theta^*) - K_{t \mid t+1}(\sigma^*)\|_F \Big)
	\end{aligned}
	\end{equation}
% }% 
with   $ \lambda_t \triangleq \| A_{t+1} \|^2 \| H_t(\sigma^*,K_{t+1}(\theta^*))\|^2, $  
$\beta_t \triangleq   \|  g_t(\sigma^*(t),K_{t+1}(\theta^*)) - K_{t \mid t+1}(\theta^*) \|_F$ and
\begin{align*}
%  &   \lambda_t \triangleq \| A_{t+1} \|^2 \| H_t(\sigma^*,K_{t+1}(\theta^*))\|^2, \\ 
%  & \beta_t \triangleq   \|  g_t(\sigma^*(t),K_{t+1}(\theta^*)) - K_{t \mid t+1}(\theta^*) \|_F \\
  &    	 H_t(\sigma^*,K_{t+1}(\theta^*)) \triangleq I-K_{t+1}(\theta^*){B_t(\sigma^*)} \\
  &  \qquad \times \big({B_t(\sigma^*)}\T K_{t+1}(\theta^*) B_t(\sigma^*) +R_t(\sigma^*)\big)^{-1}{B_t(\sigma^*)}\T.
\end{align*}  
 
% 	$$\blue	\epsilon \triangleq \sqrt{n} \ \| W_{t-1}\|_F \bigg( \sum_{t=0}^T\frac{\lambda^{t}-1}{\lambda-1}\beta_{t-1} + \sum_{t=0}^T \|K_t(\theta^*) - K_t(\sigma^*)\|_F \bigg)$$ 
% 	and   
% 	$\blue \beta_t \triangleq  \|  g_t(\sigma^*(t),K_{t+1}(\theta^*)) - K_{t \mid t+1}(\theta^*) \|_F $ 
% 	and 
% 	$\red \lambda_t \triangleq \| A_{t+1} \|_F^2 \| H_t(\sigma^*(t),K_{t+1}(\theta^*))\|_F^2$.
\end{thm}

%%%%%%%%%%%%%%%%%
% proof
%%%%%%%%%%%%%%%%%
% \begin{comment}
% \iflongversion

% \LongVersionText{

\begin{proof} 
% \blue
First, let us recall  \eqref{eq_change} and we then have
		\begin{align*}
		&\textstyle	 \sum_{t=0}^{T } {\rm tr}( K_{t}(\sigma) W_{t-1})  -   \sum_{t=0}^{T } {\rm tr}( K_{t }(\sigma^*) W_{t-1}) \\
		  = & \textstyle \sum_{t=0}^{T-1}{\rm tr }\Big( \big( K_{t|t+1}(\sigma) - K_{t|t+1}(\sigma^*) \big) \bar W_{t-1}\Big)\\
				 \leq & \textstyle \sum_{t=0}^{T -1} \|  K_{t|t+1}(\sigma) - K_{t|t+1}(\sigma^*)\|_F \| \bar W_{t-1}\|_F. 
			\end{align*}
			Next, for all $t$,   it holds that
	\begin{align}
% 	\begin{aligned} 
		& \|K_{t \mid t+1}(\sigma) - K_{t \mid t+1}(\sigma^*)\|_F \le   \|K_{t \mid t+1}(\sigma) - K_{t \mid t+1}(\theta^*)\|_F \nonumber\\  
		&\qquad + \|K_{t \mid t+1}(\theta^*) - K_{t \mid t+1}(\sigma^*)\|_F, \label{eq:difference}
		%
% 	\end{aligned}
	\end{align}
 	where $K_{t \mid t+1}(\sigma)$ is the obtained matrix when schedule $\sigma$ is used from  time   $T-1$ backwards to $t$. 
		Similarly, we define $K_t(\sigma^*)$ and $K_t(\theta^*)$.
	Furthermore, according to the definition of $\theta^*$, we have $K_{t \mid t+1}(\theta^*)=K_{t \mid t+1}^o$.
	Next, note that, due to the design of our algorithm {(line 4  in Algorithm~\ref{algm_1}),}  it holds  
	\begin{align*}
	&	 \| K_{t \mid t+1}(\sigma) - K_{t \mid t+1}(\theta^*) \|_F \\
		= &\min_i \|g_t(i,K_{t+1}(\sigma))- K_{t \mid t+1}(\theta^*)\|_F\\
		\le & \|g_t(\sigma^*(t),K_{t+1}(\sigma))- K_{t \mid t+1}(\theta^*) \|_F\\
		 \le &\|g_t(\sigma^*(t),K_{t+1}(\sigma))- g_t(\sigma^*(t),K_{t+1}(\theta^*)) \|_F  \\
		&\qquad +  \| g_t(\sigma^*(t),K_{t+1}(\theta^*)) - K_{t \mid t+1}(\theta^*) \|_F.
	\end{align*}
	\iflongversion
It then follows from Lemma~4  and  the concavity of $g_t(i,\cdot)$ that
\else
It then follows from \cite[Lemma 4]{jiao2022}  and  { the concavity of $g_t(i,\cdot)$} that
\fi
	\begin{align*}
		&\|g_t(\sigma^*(t),K_{t+1}(\sigma))- g_t(\sigma^*(t),K_{t+1}(\theta^*)) \|_F \\ 
		\le   &  \|{ H_t(\sigma^*(t),K_{t+1}(\theta^*))}\|^2 \| A_{t+1} \|^2 \\
	&\qquad 	\times \|K_{t+1 \mid t+2}(\sigma))-K_{t+1  \mid t+2}(\theta^*)\|_F .
	\end{align*}
	By defining 
% 	\begin{equation}\label{eq_beta}
% 	\begin{aligned}
$	\eta_t \triangleq 	\| K_{t \mid t+1}(\sigma) - K_{t \mid t+1}(\theta^*)\|_F ,$
$    \lambda_t \triangleq \| A_{t+1} \|^2 \| H_t(\sigma^*(t),K_{t+1}(\theta^*))\|^2,$
$ \beta_t \triangleq   \|  g_t(\sigma^*(t),K_{t+1}(\theta^*)) - K_{t \mid t+1}(\theta^*) \|_F ,$
% 	\end{aligned}
% 		\end{equation}
  we obtain
	\begin{align}\label{eta}
		\eta_t \le {\lambda}\eta_{t+1} + \beta_t, \  t=0,1,\ldots,T-2,\    \eta_{T-1} \leq \beta_{T-1}, %  \eta_T =0,
	\end{align}
	where $\lambda=\max_t\lambda_t$.
{This further gives us that 
	$  \eta_t \le  \sum_{i=1}^{T-t}\lambda^{T-t-i}\beta_{T-i}$, for $t = 0,1,\ldots, T-1$.} 
It then follows from \eqref{eq:difference}, \eqref{eta} and the definition of $\eta_t$ that
	\begin{align*}
		% 			\|K_0(\sigma) - K_0(\sigma^*)\|_F & \leq \sum_{k=0}^{T}\lambda^{T-k}\beta_{T-k},\qquad {\rm for\ } t=0\\
	&\textstyle	\sum_{t=0}^{T-1} \|K_{t \mid t+1}(\sigma) - K_{t \mid t+1}(\sigma^*)\|_F \\
	\le &\textstyle \sum_{t=0}^{T-1}\eta_t+ \sum_{t=0}^{T-1} \|K_{t \mid t+1}(\theta^*) - K_{t \mid t+1}(\sigma^*)\|_F \\
		   \le & \textstyle \sum_{t=0}^{T-1}\frac{\lambda^{t+1}-1}{\lambda-1}\beta_{t} + \sum_{t=0}^{T-1} \|K_{t \mid t+1}(\theta^*) - K_{t \mid t+1}(\sigma^*)\|_F.
	\end{align*}
This implies that 
	\begin{align*}\label{bound} 
		\sum_{t=0}^{T-1} {\rm tr}( K_{t \mid t+1}(\sigma) \bar W_{t-1})  -   \sum_{t=0}^{T-1} {\rm tr}(  K_{t \mid t+1}(\sigma^*) \bar W_{t-1})  \le  \epsilon,
	\end{align*}
	where $\epsilon$ is given in \eqref{eq_epsilon}.
	This completes the proof.
% 	\hfill $\qed$
\end{proof}
% }

% \else 

% A proof of Theorem \ref{thm_main} can be found in \cite{jiao2022}.

% \fi
% \end{comment}
%%%%%%%%%%%%%%%%%
% proof
%%%%%%%%%%%%%%%%%

\begin{remark}
Note that  equation \eqref{suboptimal_bound} in  Theorem \ref{thm_main} provides a suboptimality bound on Algorithm~\ref{algm_1}.
According to the definition of $\beta_t$, it can be seen that the value of $\epsilon$ depends on the   mismatch between the schedules $\theta^*$ and $\sigma^*$.  
Clearly, if the solution to Problem~\ref{prob_SDP} is already integer in nature (i.e., $\theta^*_t \in \{0,1\}$) for all $t$, then $\beta_t=0$ for all $t$, and consequently we obtain $\epsilon = 0$. 
%
% From the definition of $\beta_t$, we notice that $\epsilon$ captures the covariance mismatch between the schedules $\theta^*$ and $\sigma^*$, and consequently, portraying the effects of the relaxed sensor design problem on the overall optimality of the approach.  
% If  $\blue \{K_{t \mid t+1}(\sigma^*) \bar W_{t-1}\}_{t=0}^{T-1}$ is ``close" to   $\blue \{K_{t \mid t+1}(\theta^*) \bar W_{t-1}\}_{t=0}^{T-1}$ then the suboptimality bound decreases, which is expected.
% \alert{Furthermore, it is shown in \eqref{eq_epsilon} that the bound $\epsilon$ also depends on the system dimension $n$ and degrades with the system's dimension.}
\end{remark}
%}
%

% \begin{comment}
% {\blue
\section{Discussion on Multiple Actuator Scheduling and Actuation     Costs}\label{sec_discussion}

	In this section, we will provide brief discussions on the cases of   multiple actuator scheduling and   actuation costs.

%	 the case that  at each time  $N_1 (N_1 < N)$ out of  $N$ actuators \eqref{control} are used  to control the system and that actuator costs are considered. 

%
%{\red 
%	In this subsection, we provide a discussion on the multiple actuators case. We will mainly focus on two cases: (1) virtual actuators; and (2) a modified version of algorithm 1 which admits multiple actuators being selected at the same time.
%	
%	Check maybe Dipankar's paper as well.
%	
%	Inform Dipankar that indeed the algorithm  1 is now modified. Ask him to check whether the simulation results are still good.
%}

 \subsection{Multiple Actuator Scheduling}

In Section \ref{section_solution}, we   consider  the actuator scheduling problem for the case that exactly one actuator  is used at each time.
In practice, one may encounter a situation that multiple actuators (e.g., $N_t $ out of $N$) are scheduled at the same time. 
Such a problem can be solved in several ways using our method. 
Here we discuss two of them.

As a first approach, one may construct $N \choose N_t$ {\em virtual actuators}, each of these is a group of $N_t$ actuators. 
Thus, selecting $N_t$ out of $N$ actuators is equivalent to selecting one out of these  $N \choose N_t$ virtual actuators.
%Note that in this method, each virtual actuator is a  fixed collection of $N_1$ sensors.
However, complexity of such an approach grows factorially.

% A less computationally expensive approach is to use $ \sum_{i=1}^N \theta^i_t=N_t$ in Problem \ref{prob_SDP}, along with a modification in Algorithm \ref{algm_1} where instead of computing $M_t(i)$ for each $i\in \mathsf{N}$, we compute it for all possible $N_t$ actuator combinations.
 A less computationally expensive approach is to use $ \sum_{i=1}^N \theta^i_t=N_t$ in Problem \ref{prob_SDP}, along with a modification in Algorithm \ref{algm_1},
%  {\blue  
 in which case the $N_t$ actuators give the smallest values of $ \Vert K^o_{t \mid t+1}    -M_t(i) \Vert_F$ are the actuators selected at time $t$. This modification in Problem 4  does not introduce extra computational complexity.
 Computational requirements for Algorithm~1 slightly increases. However, given the simplicity of Algorithm~1, this is practically inconsequential.
%  }

\subsection{Actuation Costs}
Our main results in Section~\ref{section_solution} are derived by considering all actuators to have equal actuation costs (i.e., $c_t(i) = c$ for all $t$). 
One possible straightforward way to incorporate the actuation costs is to include the term $\sum_{t=0}^{T-1}\sum_{i=1}^N c_t(i)\theta^i_t$ in the objective function of \eqref{prob_SDP}. 
Notice that the term $\sum_{t=0}^{T-1}\sum_{i=1}^N c_t(i)\theta^i_t$ is linear in the optimization variable $\theta_t$, and hence the convexity of the problem is retained.
In the simulation we adopt this approach to include actuation costs.
%
% As we have mentioned in Section \ref{section_prob} that, in practice, however, actuators may have different energy consumption or resource usage. In this case, it is  favorable to consider also actuator costs in the actuator scheduling problem.
%
% One possible actuator cost function is   $\sum_{t=0}^{T-1} c_{\sigma(t)}$, where $c_{\sigma(t)}: {\sf N} \to \mathbb{R}_+$ and $c_i$ is the corresponding cost of the $i$-th actuator. 
%
% The objective is then to find an actuator schedule $\sigma $ that minimizes a cost function that is the sum of the finite horizon  quadratic cost  \eqref{eq_LQG_cost2} and the actuator costs $\sum_{t=0}^{T-1} c_{\sigma(t)}$, while 
% subjecting to the same constraints as in Problem \ref{prob1}.
%
% A suboptimal solution to the actuator scheduling problem with actuator costs can be obtained  using a modified version of Algorithm~\ref{algm_1}, where we now consider $ \arg\!\min_i \| K^o_{t \mid t+1} -M_t(i) \|_F + \sum_{i=1}^{\ell}c_{i} \theta^i_t - c_{\sigma}$ for selecting actuators at each time.

% }
% \end{comment}

\section{Simulation} \label{sec_simu}

\begin{figure}[t]
 	\centering
 	\begin{tikzpicture}[scale=0.6]
 		\tikzset{VertexStyle1/.style = {shape = circle,
 				color=black,
 				fill=white!93!black,
 				minimum size=1cm,
 				text = black,
 				inner sep = 1.5pt,
 				outer sep = 1pt,
 				minimum size = 0.1cm},
 			% VertexStyle2/.style = {shape = circle,
 			% 	color=black,
 			% 	fill=black!53!white,
 			% 	minimum size=0.3cm,
 			% 	text = white,
 			% 	inner sep = 2pt,
 			% 	outer sep = 1pt,
 			% 	minimum size = 0.3cm}
 		}
 	\node[VertexStyle1,draw](1) at (0,0) {$\bf 1$};
		\node[VertexStyle1,draw](2) at (-3,1.5) {$\bf 2$};
		\node[VertexStyle1,draw](3) at (-3,-1.5) {$\bf 3$};
		\node[VertexStyle1,draw](4) at (2.7,-1.5) {$\bf 4$};
		\node[VertexStyle1,draw](5) at (4.7,0) {$\bf 5$};
		\node[VertexStyle1,draw](6) at (2.7,1.5) {$\bf 6$};
  		\Edge[ style = {-,> = latex',pos = 0.2},color=black, labelstyle={inner sep=0pt}](2)(1);
 		\Edge[ style = {-,> = latex',pos = 0.2},color=black, labelstyle={inner sep=0pt}](1)(3);
 		\Edge[ style = {-,> = latex',pos = 0.2},color=black, labelstyle={inner sep=0pt}](2)(3);
 		\Edge[ style = {-,> = latex',pos = 0.2},color=black, labelstyle={inner sep=0pt}](6)(1);
		\Edge[ style = {-,> = latex',pos = 0.2},color=black, labelstyle={inner sep=0pt}](5)(1);
		\Edge[ style = {-,> = latex',pos = 0.2},color=black, labelstyle={inner sep=0pt}](4)(1);
 	\draw (2) -- (3) node [midway, above, sloped] (text1) {0.1};
		=\draw (1) -- (3) node [midway, below, sloped] (text1) {0.1};
		\draw (2) -- (1) node [midway, above, sloped] (text1) {0.2};
		% 		\draw (1) -- (2) node [midway, above, sloped] (text1) {1};
		\draw (1) -- (4) node [midway, below, sloped] (text1) {0.2};
		\draw (1) -- (5) node [midway, above, sloped] (text1) {0.1};
		\draw (1) -- (6) node [midway, above, sloped] (text1) {0.1};
 		 		% 		\Edge[ style = {<-,> = latex',pos = 0.2},color=black, labelstyle={inner sep=0pt}](1)(3);
 	\end{tikzpicture}
 	\caption{Network model for simulation example.}
 	\label{fig:sim_network}
 \end{figure}
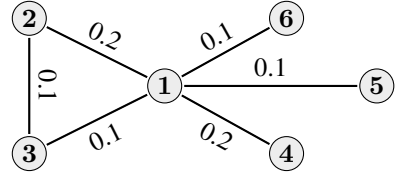

% \begin{comment}
% \begin{figure}
%     \centering
%     \includegraphics[trim = 250 50 300 10, clip,   width = 0.4 \linewidth]{network.png}
%     \caption{Network model for Simulation Example}
%     \label{fig:sim_network}
% \end{figure}
% \end{comment}

We consider a networked system with $6$ nodes as shown in Fig.~\ref{fig:sim_network}.
The $i$-th node follows the dynamics
\begin{align*}
  \textstyle  x_{t+1}(i) =  \sum_{j = 1}^6 \alpha_{ij} x_t(j) + u_t(i) + w_t(i),
\end{align*}
where $\alpha_{ij} \ge 0$ denotes the weight on the link between nodes $i$ and $j$ and $\alpha_{ii} = 1- \sum_{j = 1, j\ne i}^6 \alpha_{ij}$. 
If there is no link present between node $i$ and $j$, then $\alpha_{ij}= 0$. 
Each node has an actuator associated with it through which one can directly control the state of that node. 
The overall system state $x_t=[x_t(1),\ldots, x_t(6)]\T$ follows the dynamics 
\begin{align*}
\textstyle    x_{t+1} = A x_t + \sum_{i=1}^6 B(i) u_t(i) + w_t,
\end{align*}
where $B(i)\in \R^6$ is $i$-th canonical basis vector in $\R^6$ and $w_t=[w_t(1),\ldots, w_t(6)]\T$.
We consider a cost function of the form \eqref{eq_LQG_cost2} with $Q_t = \frac{1}{2}I, R_t(i) = I$ for all $t\le T-1$ and $Q_T = I$. 
Furthermore, we assume $x_0 \sim \N(0, \frac{1}{2}I)$ and $w_t \sim \N(0, \frac{1}{4}I)$.
The actuation costs are $c_t(i) = 1$ for $i=1,\ldots,4$, $c_t(5) = 1.5$, and $c_t(6) = 2$. 
The costs for $c_t(5)$ and $c_t(6)$ are chosen to be higher because the system is fully controllable only with $B(5)$ and $B(6)$.
For a horizon of $T = 30$, the schedule obtained from our algorithm is shown in Fig.~\ref{fig:our_schedule}, and the corresponding optimal cost is  101.0006. 

\begin{figure}[t]
    \centering
    \includegraphics[trim = 20 0 30 10, clip, height = 0.45 \linewidth]{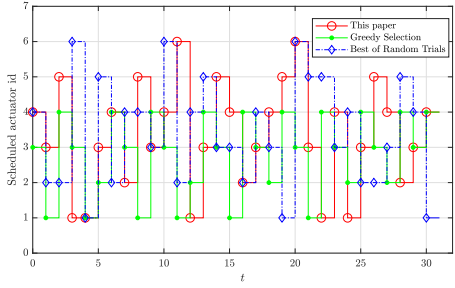}
    \caption{The actuator schedules using different methods.}
    \vspace{- 6 pt}
    \label{fig:our_schedule}
\end{figure}

Interestingly, from the solution to Problem~\ref{prob_SDP} shown in Fig.~\ref{fig:opt_schedule}, we notice that the actuator of the $1$-st node is hardly used since the values of $\theta^1_t$'s are in orders of magnitude smaller than that for the rest of the nodes for all $t$. 
This is in contrast with the schedule we found in Fig.~\ref{fig:our_schedule} where actuator~$1$ is scheduled for several time instances ($\sim 17$\% of the time) by Algorithm~\ref{algm_1}. 
Actuator 2 is used the least by Algorithm~\ref{algm_1} in Fig.~\ref{fig:our_schedule}, however, in Fig.~\ref{fig:opt_schedule} we notice that $\theta^2_t$ is not the least among all $\theta^i_t$'s.
While one might be tempted to only use actuators 3 and 4 since the corresponding $\theta^i_t$ values are the highest ones in Fig.~\ref{fig:opt_schedule}, however, such restriction leads to a cost of 108.5531, which is higher than what our method found.
This indeed validates our statements in Remark \ref{rem_actuator_selection}.
% Furthermore, from Fig.~\ref{fig:opt_schedule} we notice that the contributions from the actuators of nodes $5$ and $6$ are the most since, for any time $t$, $\theta^5_t$ (and $\theta^6_t$) has the highest value. 
% This was tempting to test whether scheduling the actuators $5$ and/or $6$ for all the time would produce a smaller cost than what we have obtained. 
% However, it was not the case. 
% The performance cost from using \textit{only} the $5$-th or the $6$-th actuator was 61.6688.
% Then we performed another experiment allowing the system to choose between the $5$-th or the $6$-th actuators.
% In that case, our algorithm found the schedule that alternately activated the two actuators and the cost was 51.5505, which is still higher than the cost 44.5984 found from the schedule of Fig.~\ref{fig:our_schedule}. 
% Therefore, while it might be tempting to schedule the actuator corresponding to the maximum $\theta^i_t$ from the solution of Problem~\ref{prob_SDP} (e.g., shown in Fig.~\ref{fig:opt_schedule}), however, such a rounding technique is not necessarily useful, as demonstrated in this example.

\begin{figure}[t]
    \centering
    \includegraphics[trim = 20 0 30 10, clip, height = 0.45 \linewidth]{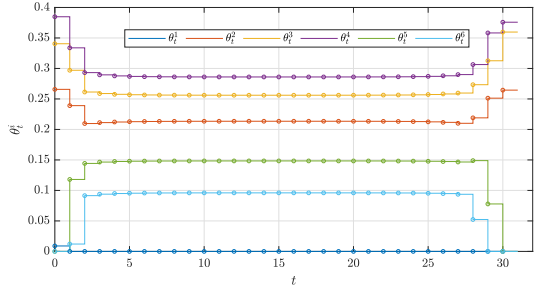}
    \caption{Optimal $\theta^i_t$ from Problem~\ref{prob_SDP}. For all $t$, we obtained $\theta^5_t = \theta^6_t$, and hence they are not distinguishable in the figure.}
    \vspace{- 6 pt}
    \label{fig:opt_schedule}
\end{figure}

Next, in order to compare the performance of our proposed approach
with randomly generated schedules and  a greedy selection approach.\footnote{
Scheduling problems generally has a supermodularity structure which ensures a level of optimality guarantee for the greedy approach. }
We randomly selected 50,000 schedules and computed the cost corresponding to these schedules. 
The resulting cost distribution from the schedules are plotted in Fig.~\ref{fig:histogram} and the minimum cost out of these 50,000 trials is 102.0693.

\begin{figure}[!t]
    \centering
    \includegraphics[ height = 0.42 \linewidth]{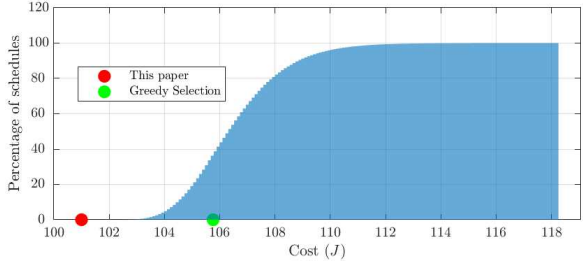}
    \caption{x-axis: Cost ($J$), y-axis: Percentage of the randomly generated trials which produced a cost less than or equal to the value on the x-axis.}
    \vspace{- 12 pt}
    \label{fig:histogram}
\end{figure}
Evaluation of the 50,000 random trials took 34.65 seconds whereas our approach (convex optimization plus trajectory tracking) took 2.5 seconds, which is an order of magnitude less time.
For the greedy approach, at each time instance we greedily selected the actuator that provides the minimum cost for that time stage. This approach is fast ($< 1$ sec) but the performance is the worst (see. Fig.~\ref{fig:histogram}).

% \ifnewversion
% {
\section{Conclusions}\label{sec_conclusion}
In this letter, we have studied the problem of  actuator scheduling  for  stochastic  linear NCSs. 
In particular, we have considered the case that only one actuator is active at each time.
We have proposed a convex relaxation and used its solution as a reference for obtaining a  suboptimal tracking algorithm for solving the actuator scheduling problem.
Suboptimality guarantees for the proposed algorithm have been provided using dynamic programming arguments.
%
% {\blue
We have also discussed the extensions on  the cases with multiple actuator scheduling and  actuation costs.
% }
% \balance
% \else

% {\red
% \section{Conclusions}\label{sec_conclusion}
% In this paper, we have studied the problem of  actuator scheduling  for {\red stochastic} linear {\red networked control} systems. 
% %
% In particular, we have considered the case that only one actuator is active at each time.
% %
% We have proposed a suboptimal algorithm for solving the actuator scheduling problem by studying its connection to an LQG control problem with constraints. 
% %
% Suboptimality guarantees for the proposed algorithm have been provided using dynamic programming arguments.
% %
% % {\blue
% We have also discussed the extensions on  the cases with multiple actuator scheduling and  actuation cost.
% }

% \fi
% }

\balance

\bibliographystyle{IEEEtran}
\bibliography{arXiv.bib}    
% \bibliography{abbrv.bib}

\end{document}